\documentclass[12pt]{amsart}
\usepackage{amscd,amssymb}
\usepackage{graphicx,enumerate}

\usepackage{color}
\usepackage[arrow,matrix,arrow,ps,color,line,curve,frame]{xy}
\usepackage{pgf, tikz}
\usepackage{url}
\usepackage{enumerate}

\makeatletter
\def\url@leostyle{%
  \@ifundefined{selectfont}{\def\UrlFont{\sf}}{\def\UrlFont{\small\ttfamily}}}
\makeatother
\usetikzlibrary{matrix,arrows}

\let\oldlabel=\label

\def\prellabel{\marginparsep=1em
    \def\label##1{\oldlabel{##1}\ifmmode\else\ifinner\else
         \marginpar{{\footnotesize\ \\ \tt
                    ##1}}\fi\fi}}

\def\int{\operatorname{int}}
\def\conv{\operatorname{conv}}
\def\vertex{\operatorname{vert}}

\def\rest{\operatorname{rest}}
\def\ess{\operatorname{ess}}

\def\Hom{\operatorname{Hom}}
\def\pol{\operatorname{Pol}}
\def\Aff{\operatorname{Aff}}
\def\join{\operatorname{join}}
\def\C{\operatorname{C}}
\def\o{\operatorname{o}}
\def\op{\operatorname{op}}

\def\Sets{\mathbf{Sets}}
\def\Poly{\mathbf{Pol^\copyright}}
\def\Cones{\mathbf{Cones^\copyright}}
\def\Pol{\mathbf{Pol}}
\def\cD{{\mathcal D}}

\def\RR{{\mathbb R}}

\def\ZZ{{\mathbb Z}}
\def\NN{{\mathbb N}}
\def\AA{{\mathbb A}}

\let\phi=\varphi
\let\theta=\vartheta

\advance\headheight1.15pt

\newtheorem{lemma}{Lemma}[section]
\newtheorem{corollary}[lemma]{Corollary}
\newtheorem{theorem}[lemma]{Theorem}
\newtheorem{proposition}[lemma]{Proposition}

\theoremstyle{definition}
\newtheorem{definition}[lemma]{Definition}
\newtheorem{remark}[lemma]{Remark}
\newtheorem{example}[lemma]{Example}

\begin{document}

\title[Affine hom-complexes]{Affine hom-complexes}

\begin{abstract}
For two general polytopal complexes the set of face-wise affine maps between them is shown to be a polytopal complex in an algorithmic way. The resulting algorithm for computing the affine hom-complex is analyzed in detail. There is also a natural tensor product of polytopal complexes, which is the left adjoint functor for $\Hom$. This extends the corresponding facts from single polytopes, systematic study of which was initiated in \cite{Hompolytopes,Homregpolytopes}. Explicit examples of computations of the resulting structures are included. In the special case of simplicial complexes, the affine hom-complex is a functorial subcomplex of Kozlov's combinatorial hom-complex \cite{Kozlov}, which generalizes Lov\'asz' well-known construction \cite{Lovasz} for graphs.
\end{abstract}

\author{M. Bakuradze, A. Gamkrelidze, and J. Gubeladze}

\address{Department of Mathematics\\
Faculty of Exact and Natural Sciences\\
         Tbilisi State University\\
         3 Chav\-cha\-vadze Ave., Tbilisi 0179, Georgia}
\email{malkhaz.bakuradze@tsu.ge}

\address{Department of Computer Sciences\\
Faculty of Exact and Natural Sciences\\
         Tbilisi State University\\
         3 Chavchavadze Ave., Tbilisi 0179, Georgia}
\email{alexander.gamkrelidze@tsu.ge}

\address{Department of Mathematics\\
         San Francisco State University\\
         1600 Holloway Ave.\\
         San Francisco, CA 94132, USA}
\email{soso@sfsu.edu}

\thanks{The authors were supported by Georgian NSF grant DI/16/5-103/12; Gubeladze was also supported by NSF grant DMS 1301487}




\maketitle

\section{Introduction}

For two (convex) polytopes $P$ and $Q$ the set of affine maps between them is a polytope in a natural way. We denote it by $\Hom(P,Q)$. The software \textsf{Polymake} \cite{polymake,Joswig} has a module for computing these \emph{hom-polytopes}. A systematic theory for hom-polytopes was initiated in \cite{Hompolytopes}. In general, the construction $\Hom(P,Q)$ is very fragile in the sense that a small `perturbation' of the input polytopes $P$ and $Q$ may change even the combinatorial type of the hom-polytope. Our treatment of the category of polytopes and their affine maps, denoted by $\Pol$, is motivated by a conjectural fusion of algebraic and geometric aspects of $\Pol$ into a homological theory of polytopes. One of the initial observations here is that there is a symmetric \emph{tensor product} of polytopes, satisfying the usual conjunction $\otimes\dashv\Hom$, and there is a parallel theory for cones  \cite{Hompolytopes,Valby}.

Determination of the facets of $\Hom(P,Q)$ is straightforward; see Section \ref{Basic}. On the other extreme, determination of the vertices of $\Hom(P,Q)$ is a real challenge \cite{Hompolytopes}. Simplices are the \emph{free} objects in $\Pol$: every map from the vertices of a simplex $\sigma$ to a polytope $P$ extends uniquely to an affine map $\sigma\to P$, i.e., $\Hom(\sigma,P)=P^{\#\vertex(\sigma)}$. However, just outside the class of simplices, already for high dimensional cubes and cross-polytopes, one is lead to surprisingly rich combinatorics \cite{Homregpolytopes}.

One can trace the relevance of the concept of polytopal hom-objects, without introducing them explicitly, to triangulation theory \cite{Fiber}, statistics \cite{Sandwiched}, and quantum theory \cite{Physics}. More recently, the hom and tensor functors for general convex cones in the context of nonlinear optimization problems were studied in \cite{Velasco}. 

The categories of polytopes $\Pol$ is a full subcategory of the category of polytopal complexes and their affine maps, which we denote by $\Poly$. Here an affine map between two complexes means a map between the support spaces which is affine on each face.

We show that for two general polytopal complexes $\Pi_1$ and $\Pi_2$ the set of face-wise affine maps between them is a polytopal complex and there is a \emph{tensor product complex} $\Pi_1\otimes\Pi_2$ so that the two constructions form a pair of adjoint functors. This is proved in Theorem \ref{Existence-theorem} and Corollary \ref{conjunction}. In deriving these properties, the embedding $\Poly$ into the category of \emph{conical complexes} $\Cones$ via the \emph{coning} functor is very useful; Section \ref{Coning}.

The algorithmic nature of Theorem \ref{Existence-theorem} is analyzed in detail in Section \ref{Implementing}, where several challenges along the way of implementation of the algorithm are addressed. In particular, for \emph{Euclidean} complexes the process of determination of the hom-complex can be substantially simplified. Furthermore, the hom-complex between simplicial complexes is completely transparent -- it turns out to be a functorial subcomplex of Kozlov\rq{}s combinatorial hom-complex between simplicial complexes \cite[Ch.9.2]{Kozlov} -- a higher dimensional version of Lov\'asz\rq{}  hom-complex between graphs, implicit in \cite{Lovasz} and explicitly introduced later in \cite{Kozlov-Israel}. The latter concept has initiated the subject of topological methods in the study of chromatic numbers of graphs; see \cite{Lovasz-conjecture,Dochter,Kozlov} .

\bigskip\noindent\emph{Acknowledgment.} We thank anonymous reviewers for their comments that lead to substantial improvement of the exposition.

\section{Basic facts}\label{Basic-facts}

All polytopes in this paper are assumed to be convex. Our references for polytopes are  \cite[Ch.1]{Kripo} and \cite{Ziegler}.


For the reader's convenience we now summarize basic definitions and facts.

\subsection{Affine spaces and maps}\label{Affinespaces} All our vector spaces are real and finite dimensional. An \emph{affine space} is a parallel translate of a vector subspace, i.e., a subset of the form $H=x+V'\subset V$, where $V$ is a vector space, $V'\subset V$ a subspace, and $x\in V$. A map between two affine spaces $f:H_1\to H_2$ is \emph{affine} if $f$ respects \emph{barycentric coordinates}. Equivalently, if $H_1\subset V_1$ and $H_2\subset V_2$ are ambient vector spaces, then $f:H_1\to H_2$ is affine if there is a linear map $\alpha:V_1\to V_2$ and an element $x\in V_2$ such that $f$ is the restriction of $\alpha(-)+x$ to $H_1$.

For a subset $X\subset\RR^n$, denote by:

\begin{enumerate}[{\rm $\centerdot$}]
\item $\RR X$ the linear span of $X$,
\item $\conv(X)$ the convex hull of $X$,
\item $\Aff(X)$ the affine hull of $X$,
\end{enumerate}

The set of affine maps between two affine spaces $H_1$ and $H_2$ will be denoted by $\Aff(H_1,H_2)$.

Let $H_1\subset V_1$ and $H_2\subset V_2$ be affine subspaces in their ambient vector spaces. Upon fixing an affine surjective map $\pi:V_1\to H_1$, which restricts to the identity map on $H_1$, we get the injective map
$$
\theta_\pi:\Aff(H_1,H_2)\to\Aff(V_1,V_2),\quad f\mapsto \iota\circ f\circ\pi,
$$
where $\iota:H_2\hookrightarrow V_2$ is the inclusion map. We also have the embedding into the space of linear maps:
\begin{align*}
\theta:\Aff(V_1,V_2)\to&\Hom(V_1\oplus\RR,V_2\oplus\RR),\\
&(\theta(f))(cx,c)=(cf(x),c),\\
&(\theta(f))(x,0)=(f(x)-f(0),0),\\
&\quad\qquad\qquad\qquad f\in\Aff(V_1,V_2),\quad x\in V_1,\quad c\in\RR.
\end{align*}
The composite map $\theta\circ\theta_\pi$ identifies $\Aff(H_1,H_2)$ with an affine subspace of $\Hom(V_1\oplus\RR,V_2\oplus\RR)$ and the induced convexity notion in $\Aff(H_1,H_2)$ is independent of the choice of $\pi$.

\subsection{Polytopes and cones}\label{Polytopes-and-cones} A \emph{polyhedron} is the intersection of finitely many closed halfspaces, i.e., the solution set to a finite system of not necessarily homogeneous linear inequalities. A \emph{polytope} always means a convex polytope in an ambient vector (or affine) space, i.e., polytopes are the bounded polyhedra. An \emph{affine map} between two polyhedra is the restriction of an affine map between the ambient spaces.

For two polytopes $P$ and $Q$ we write $P\cong Q$ if the polytopes are isomorphic objects in the category of polytopes and affine maps $\Pol$.

For two polytopes in their ambient vector spaces $P\subset V$ and $Q\subset W$ let:

\begin{enumerate}[{\rm $\centerdot$}]
\item $\Hom(P,Q)$ denote the set of affine maps $P\to Q$;
\item $\Aff(P,Q)=\Aff(\Aff(P),\Aff(Q))$;
\item $\vertex(P)$ denote the set of vertices of $P$;
\item $\int(P)$ denote the relative interior of $P$ in $\Aff(P)$;
\item $\partial P=P\setminus\int(P)$, the boundary of $P$;
\item $\join(P,Q)=\conv\big((P,0,0),(0,Q,1)\big)\subset V\oplus W\oplus\RR$.
\item $P\otimes Q=\conv\big((v\otimes w,v,w)\ :\ v\in\vertex(P),\ w\in\vertex(Q)\big)\subset(V\otimes W)\oplus V\oplus W$.
\end{enumerate}

For generalities on cones see \cite[Chapter 1]{Kripo} and \cite[Chapter 1]{Ziegler}.

The set of nonnegative reals is denoted by $\RR_+$. For a subset $X\subset\RR^n$, the set $\RR_+ X$ of non-negative real linear combinations of finitely many elements of $X$ is called the \emph{conical hull} of $X$. A \emph{cone} in this paper means a finite polyhedral \emph{pointed} cone, i.e., the conical hull  of a finite subset of $\RR^n$, containing no non-zero linear subspace. Equivalently, a cone is a polyhedron defined by homogeneous systems of linear inequalities and containing no non-zero subspace.  For a cone $C$ the \emph{dual conical set} $C^{\o}=\{x\in\RR^n\ |\ x\cdot y\ge0\ \text{for all}\ y\in C\}\subset\RR^n$ is a cone if and only if $\dim C=n$. Here $x\cdot y$ is the dot-product. If $\dim C=n$ then $C^{\o}$ is called the \emph{dual cone} for $C$.

An affine map of cones is always assumed to be linear, i.e., mapping $0$ to $0$. The set of affine maps between two cones $C$ and $D$ will be denoted by $\Hom(C,D)$.

The \emph{tensor product} of two cones in their ambient vector spaces $C\subset V$ and $D\subset W$ is defined to be the cone
$$
C\otimes D=\RR_+\{x\otimes y\ :\ x\in C,\ y\in D\}\subset V\otimes W.
$$

A \emph{facet} of a polytope or cone is a maximal proper face.

\subsection{Limits}\label{Limits}

The standard reference for basic categorial concepts is \cite{Categories}. For generalities on enriched categories see \cite{Enriched}. Several of these concepts in a geometric setting can be found in \cite{Kozlov}. The reader needs no background in category theory though because every categorial notion (e.g., functor, (co)limit, conjunction, monoidal structure), used in the paper, is eventually explained in geometric terms.

Because of its importance in the proof of the main existence result (Theorem \ref{Existence-theorem}), we recall how the limit of a finite diagram in $\Pol$ is determined. Let $\cD$ be such a diagram, i.e., a family of finitely many polytopes $P_1,\ldots,P_k$ and finitely many affine maps between them. Then
\begin{align*}
\lim_{\leftarrow}\cD=\big\{(x_1,\ldots,x_k)\in P_1\times\cdots\times P_k\ :\ &f(x_i)=x_j\ \text{for any}\\ 
&f:P_i\to\ P_j\ \text{in}\ \cD\big\}.
\end{align*}
This limit is a polytope in the limit affine space, resulting from the corresponding diagram of affine hulls $\Aff(P_1),\ldots,\Aff(P_k)$ and affine maps, the latter limit being determined similarly.

\subsection{Basic facts}\label{Basic}

The following theorem encapsulates basic facts on polytopes and cones; see \cite[\S2--3]{Hompolytopes} for details.

\begin{theorem}\label{folklore} Let $P$ and $Q$ be polytopes and let $C,D,E$ be cones.
\begin{enumerate}[{\rm (a)}]
\item The set $\Hom(C,D)$ is a $(\dim C\cdot\dim D)$-dimensional cone in the vector space $\Hom(\RR C,\RR D)$.
\item The extremal rays of $C\otimes D$ are the tensor products of the extremal rays of $C$ and $D$.
\item For faces $C'\subset C$ and $D'\subset D$ we have the face $C'\otimes D'\subset C\otimes D$. In general, $C\otimes D$ has many other faces.
\item The following map is a linear bijection
\begin{align*}
\Theta:\Hom(C\otimes D,E)\to\Hom(C,\Hom(D,E)),\quad\big(\Theta(\alpha)(x)\big)(y)=\alpha(x\otimes y).
\end{align*}
\item The set $\Hom(P,Q)$ naturally embeds as a polytope into the affine space $\Aff(P,Q)$.
\item $\dim(\Hom(P,Q))=\dim P\dim Q+\dim Q$.
\item The facets of $\Hom(P,Q)$ are the subsets of the form
$$
H(v,F)=\{f\in\Hom(P,Q\}\ |\ f(v)\in F\},
$$
where $v\in P$ is a vertex and $F\subset Q$ is a facet.
\item For every vertex $w\in Q$, the map $f:P\to Q$, $f(P)=\{w\}$, is a vertex of $\Hom(P,Q)$. In general, $\Hom(P,Q)$ has many other vertices.
\item $\Hom(P\otimes Q, R)\cong\Hom(P,\Hom(Q,R))$.
\item
If $\sigma$ is an $n$-dimensional simplex then
\begin{align*}
\Hom(\sigma,P)\cong P^{n+1},\qquad \sigma\otimes P\cong\join\big(\underbrace{P,\ldots,P}_{n+1}\big)
\end{align*}
($n$-fold iteration of $\join(-,P)$, applied to $P$).
\item If $\cD$ is the diagram in $\Pol$, consisting of $P$ and $Q$ and no affine map, then
$$
P\times Q=\lim_{\leftarrow}\cD,\qquad\join(P,Q)=\lim_{\to}\cD.
$$
\end{enumerate}
\end{theorem}
\noindent(We have omitted the obvious cone analogs of (g).)

Theorem \ref{folklore} in particular says that, for two polytopes $P$ and $Q$, the facets of $\Hom(P,Q)$ and the vertices of $P\otimes Q$ are straightforward. The works \cite{Hompolytopes,Homregpolytopes} explore the vertices of $\Hom(P,Q)$ in various situations. Similarly, for two cones $C$ and $D$ the facets of $\Hom(C,D)$ and the extremal rays of $C\otimes D$ are straightforward and the challenge is to understand the extremal rays of the former and the facets of the latter.

\section{Affine polyhedral complexes}\label{Existence}

\subsection{Complexes and their affine maps}\label{Polyhedral}
We start with the following general definition.

\begin{definition}\label{polyhedral}
An \emph{affine polyhedral complex} consists of (i) a finite family $\Pi$ of nonempty sets, called \emph{faces}, (ii)
a family $P_p$, $p \in \Pi$, of polyhedra, and (iii) a family $\pi_p:P_p \to p$ of bijections
satisfying the following conditions:
\begin{enumerate}[{\rm (a)}]
\item for each face $F$ of $P_p$, $p\in \Pi$, there exists $f \in \Pi$ with $\pi_p(F)=f$;
\item for all $p,q \in \Pi$ there exist faces $F$ of $P_p$ and $G$ of $P_q$ such that
$p \bigcap q =\pi_p(F)=\pi_q(G)$ and, furthermore, the restriction of $\pi_q^{-1}\circ \pi_p$ to $F$ is an affine isomorphism of the polyhedra $F$ and $G$.
\end{enumerate}

We denote an affine polyhedral complex simply by $\Pi$, assuming that the polyhedra $P_p$
and the maps $\pi_p\in \Pi$, are clear from the context.

$|\Pi|$ stands for the \emph{support space} $\bigcup_{p\in\Pi}p$ of $\Pi$, with the induced topology.

The maximal faces of $\Pi$ will be called its \emph{facets}.

We speak of a \emph{polytopal complex} if the polyhedra $P_p$ are polytopes and a \emph{conical complex} if the polyhedra $P_p$ are cones.

A \emph{simplicial complex} is polytopal complex whose faces are simplices.
\end{definition}
\noindent (Conical complexes are called \emph{weak fans} in \cite{Plycom}.)

The level of generality of polyhedral complexes in algebraic/topological combinatorics varies from Euclidean complexes (e.g., \cite{Stanley}), defined below, to even broader classes than affine polyhedral complexes, closer to \emph{regular $CW$-complexes} (e.g., \cite{Kozlov}). Since in this work we consider only affine polyhedral complexes we will suppress `affine'.

\medskip One has the following hierarchy of polytopal complexes:

\begin{align*}
&\text{Simplicial complexes}\subsetneq\\
&\hspace{4em}\text{Boundary complexes}\subsetneq\\
&\hspace{8em}\text{Euclidean complexes}\subsetneq\\
&\hspace{12em}\text{Polyhedral complexes}.
\end{align*}

\medskip\noindent Here a \emph{boundary complex} refers to a subcomplex of the full face poset of a single polytope and an \emph{Euclidean complex} refers to a complex whose support space admits an embedding into a vector space which is face-wise affine.

A more refined hierarchy of polytopal complexes in the presence of face-wise \emph{lattice structures} (in the sense of the integer lattice $\ZZ^d$) was considered in \cite{Plycom} where the automorphism groups of the associated arrangements of toric varieties were studied -- graded automorphisms in the affine case and full groups in the projective case.

The following examples illustrate the proper embeddings in the hierarchy above.

\begin{example}\label{complexes-examples}
In Figure 1, the complex $\Pi_a$ has six copies of the unit square $[0,1]^2$, forming the boundary of the unit cube $[0,1]^3$, and one big diagonal of $[0,1]^3$, glued together along common faces as shown. The complexes $\Pi_b$ and $\Pi_c$ have, respectively, three and four copies of $[0,1]^2$ as facets, glued together along common edges.

The complex $\Pi_a$ is obviously Euclidean; but it is not boundary. In fact, if $\partial([0,1]^3)$ was embedded into the face complex of a polytope $Q\subset\RR^d$, then the big diagonal would necessarily pierce the interior of the 3-dimensional sub-polytope $\conv(\partial([0,1]^3))\subset Q$, making impossible for this diagonal to be an edge of $Q$.

The complexes $\Pi_b$ and $\Pi_c$ are not even Euclidean. In fact, if the \emph{polytopal M\"obius strip} $\Pi_b$ was Euclidean then the three parallel segments, along which the squares are glued, would have same orientation. If $\Pi_b$ was Euclidean then the edges of the right visually non-distorted square would be parallel, forcing the square to collapse into a segment.

\begin{figure}[h!]
\caption{Three polytopal complexes}
\vspace{.15in}
\includegraphics[trim = 0mm 0.1in 0mm 0.1in, clip, scale=.25]{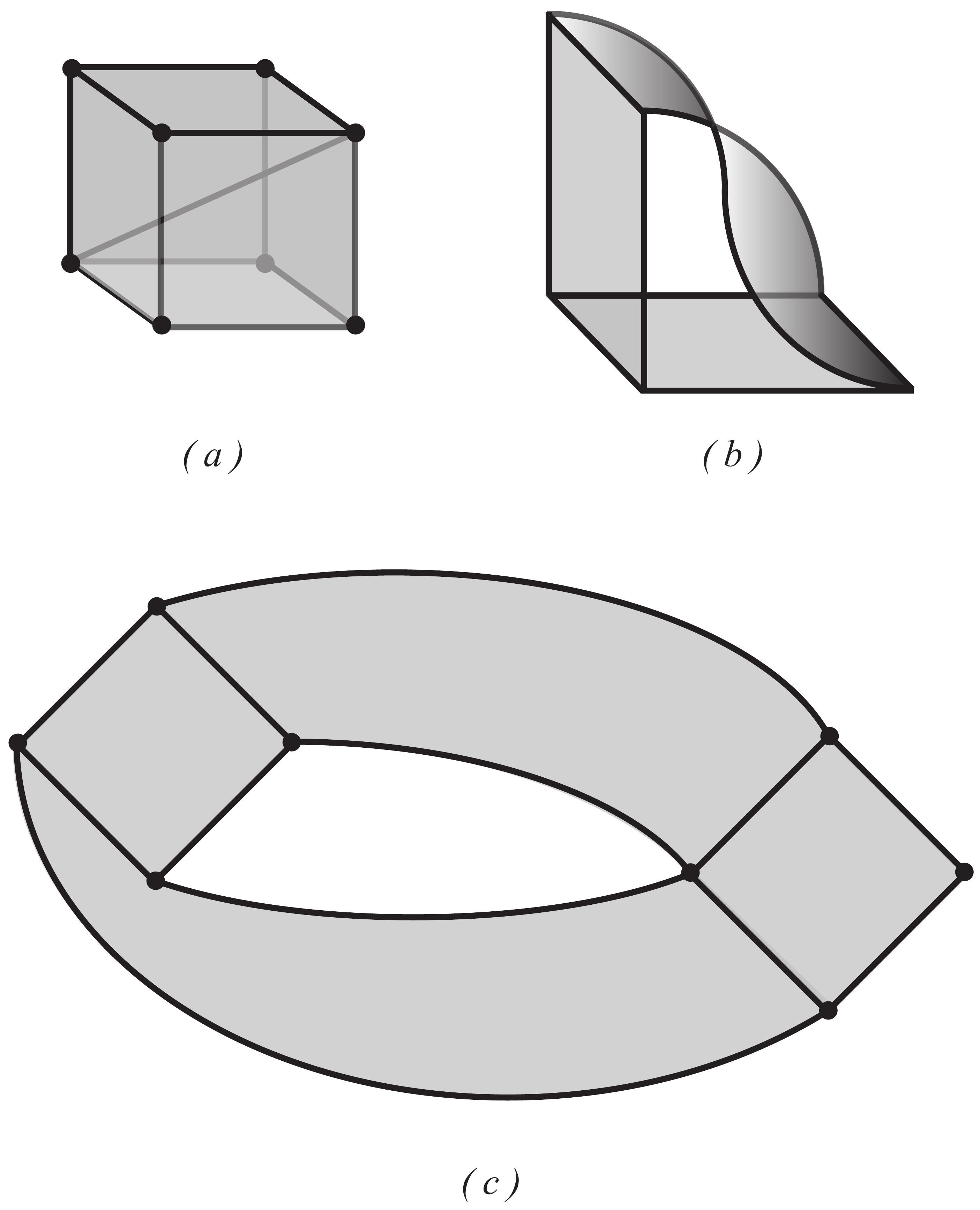}
\end{figure}
\end{example}

\begin{definition}\label{complex-maps}
Let $\Pi_1$ and $\Pi_2$ be two polyhedral complexes. An \emph{affine map} $\Pi_1\to\Pi_2$ is a map
$f:|\Pi_1|\to|\Pi_2|$, such that for every $p\in\Pi_1$ there exists $q\in\Pi_2$, satisfying the conditions $f(p)\subset q$ and $\pi_q\circ f\circ\pi_p^{-1}:P_p\to P_q$ is an affine map.  The set of affine maps $\Pi_1\to\Pi_2$ will be denoted by $\Hom(\Pi_1,\Pi_2)$ and called the \emph{affine hom-complex between $\Pi_1$ and $\Pi_2$}. The set $\Hom(\Pi_1,\Pi_2)$ is a subspace of the topological space of continuous maps $|\Pi_1|\to|\Pi_2|$ with the compact-open topology.
\end{definition}
\noindent (To keep notation simple, in the composition above the same $\pi$ is used for the structural maps of $\Pi_1$ and $\Pi_2$.)

\subsection{A topological computation}\label{Topological-computation} In the table below we describe the spaces
$$
\Hom(\Pi_i,\Pi_j),\quad i,j\in\{a,b,c\},
$$
for the polytopal complexes in Example \ref{complexes-examples} (Figure 3).  More precisely, we list the connected components up to affine isomorphism of polytopes or strong deformation retraction. For instance, the equality
$$
\Hom(\Pi_b,\Pi_b)=\{6\ \text{segments, retract}\ \Pi_b\}
$$
means that the affine hom-complex in question has seven connected components of which six are homeomorphic to the segment $[0,1]$ and $|\Pi_b|$ is a strong deformation retract of the seventh. Moreover, it will be shown in Section \ref{Hom-complex} that the affine hom-complexes are polytopal complexes; in the special case of $\Hom(\Pi_b,\Pi_b)$ the indicated connected components turn out to be segments.

The entry in the $\Pi_i$-th row and $\Pi_j$-th column refers to $\Hom(\Pi_i,\Pi_j)$:

\bigskip

\resizebox{0.95\textwidth}{!}{\begin{minipage}{\textwidth}
 \begin{tabular}{ l || c | c| r }
    \hline\hline
    $\Hom(-,-)$  & $\Pi_a$ & $\Pi_b$ & $\Pi_c$ \\ \hline\hline
    $\Pi_a$ & 12 points, retract $\Pi_a$ & retract $\Pi_b$ & retract $\Pi_c$ \\ \hline
    $\Pi_b$ & retract $\Pi_a$ & 6 segments, retract $\Pi_b$ & retract $\Pi_c$ \\ \hline
    $\Pi_c$ & retract $\Pi_a$ & retract $\Pi_b$ & 2 segments, retract $\Pi_c$
  \end{tabular}
\end{minipage}}

\bigskip

Here is our argument.

First we observe the following \emph{incommensurability property,} which is rather obvious from topological/combinatoral considerations: for any pair $i,j\in\{a,b,c\}$, $i\not=j$, any affine map $\Pi_i\to\Pi_j$ maps all of $|\Pi_i|$ to a facet of $\Pi_j$.

Pick $i,j\in\{a,b,c\}$, $i\not=j$. Any face $P\in\Pi_j$ is a strong deformation retract of $\Hom(\Pi_i,P)$ -- the set of affine maps $\Pi_i\to\Pi_j$ evaluating in $P$. This is shown as follows: $P$ embeds into $\Hom(\Pi_i,P)$ via $x\mapsto f$, where $f(|\Pi_i|)=x$, and for any $g\in\Hom(\Pi_i,P)$ we have the continuous family $\{g_t\}_{[0,1]}\subset\Hom(\Pi_i,P)$, where $g_t$ is the composition of $g$ with the homothety of $P$, centered at the barycenter of $g(|\Pi_i|)$ and with coefficient $1-t$. Consequently, $\Hom(\Pi_i,\Pi_j)$ is covered by closed subspaces, indexed by the facets of $\Pi_j$, each containing the corresponding facet as a strong deformation retract. Moreover, the intersection $X$ of any subfamily of these spaces is indexed by the corresponding intersection of facets of $\Pi_j$, the latter sitting inside $X$ as a strong deformation retract, and the involved strong deformation retractions are all compatible. This explains the non-diagonal entries in the table.

As for the diagonal entries, we first separate the connected components formed by the affine maps which evaluate in single facets. Here the same argument we used for the non-diagonal entries produces the indicated strong deformation retracts.

Finally, the remaining components are accounted for as follows.

An affine map $f:\Pi_a\to\Pi_a$, not evaluating in a facet, must be an automorphism. Such an $f$ extends to a unique affine automorphism of the cube, bounded by the six squares. But $f$ must also map the spatial diagonal onto itself. If the diagonal maps identically to itself, then $f$ is uniquely determined by a permutation of the three edges, adjacent to one of the end-points of the diagonal. Hence six isolated points, and six more correspond to those $f$-s that invert the spatial diagonal.

Let $E_1,E_2,E_3$ be the three parallel edges in $\Pi_b$ along which the facet squares are glued together. An affine map $f:\Pi_b\to\Pi_b$, not evaluating in a facet, must map each of these edges to an edge from the same set. There are six possibilities:
\begin{align*}
E_1\mapsto E_{\sigma(1)},\ E_2\mapsto E_{\sigma(2)},\ E_3\mapsto &E_{\sigma(3)},\qquad \sigma\in S_3,
\end{align*}
defining six mutually homeomorphic connected components of $\Hom(\Pi_b,\Pi_b)$. It is enough to consider the case $\sigma={\bf 1}$. Clearly, $f$ is uniquely determined by its restriction to the end-points of $E_1$. We can think of $E_1$ as $[0,1]$. Moreover, the M\"obius strip structure implies $f(0)=1-f(1)$. So the map $f$ is completely determined by the value $f(0)$ which can be any point in $[0,1]$. So the connected component is a segment.

An affine map $f:\Pi_c\to\Pi_c$, not evaluating in a facet, must map the left and right squares to themselves and the other two squares either to themselves or to each other. This follows from keeping track of the vertex/edge/facet incidences. So we have two homeomorphic connected components and it is enough to characterize the connected component, containing the identity map ${\bf1}:\Pi_c\to\Pi_c$. Since the left vertex of the right square can not be perturbed continuously so that it remains an element of the three adjacent facets, the only way the identity map can be continuously perturbed is via sliding the upper vertex of the same square along the south-west edge. The complex $\Pi_c$ is such that every point in this edge, thought of as the image of the mentioned upper vertex under $f$, uniquely determines the whole map $f$. So the connected component of the identity map is a segment.

\medskip We remark that the determination of the \emph{polytopal complex} structures on the affine hom-complexes $\Hom(\Pi_i,\Pi_j)$, due to the huge size, is  beyond reach unless one actually implements the algorithms which will be introduced in Sections \ref{Hom-complex} and \ref{Implementing}.

The argument we used above to describe deformation retracts of certain connected components works for arbitrary polytopal complexes. This leads to a general result which is interesting even for single polytopes and hence worth of writing up:

\begin{proposition}\label{deformation-retract}
Let $\Pi_1$ and $\Pi_2$ be polytopal complexes. Then $|\Pi_2|$ is a strong deformation retract of a connected component of $\Hom(\Pi_1,\Pi_2)$. Moreover, the deformation retraction can be chosen to be affine in an appropriate sense. In particular, for any two polytopes $P$ and $Q$, there is an affine embedding $\iota:Q\to\Hom(P,Q)$ and an affine map $h:\Hom(P,Q)\times[0,1]\to\Hom(P,Q)$, such that
\begin{enumerate}[{\rm (i)}]
\item $h(-,0)$ is the identity map of $\Hom(P,Q)$,
\item $h(-,1):\Hom(P,Q)\to\iota(Q)$,
\item $h(-,t)$ is the identity map on $\iota(Q)$ for every $t\in[0,1]$.
\end{enumerate}
\end{proposition}

\section{Coning and tensor product}\label{Coning}

Polyhedral complexes and affine maps form a category. Denote by $\Poly$ and $\Cones$ the subcategories of polytopal and conical complexes and their affine maps, respectively.

Let $\Pi$ be a polytopal complex. For $p\in\Pi$ let $P_p\subset E_p$ be the ambient vector space. Then the cones $\RR_+(P_p,1)\subset E_p\oplus\RR$, $p\in\Pi$, assemble into a conical complex which we denote by $\C(\Pi)$. The complex $\Pi$ can be thought of as the `cross section of $\C(\Pi)$ at height 1'. In particular, we can assume $|\Pi|\subset|\C(\Pi)|$. Every affine map between two polytopal complexes $f:\Pi_1\to\Pi_2$ extends uniquely to an affine map $\C(\Pi_1)\to\C(\Pi_2)$ and we get the \emph{coning} (or, \emph{homogenization}) functor:
$$
\C:\Poly\to\Cones.
$$

\begin{example}
\begin{enumerate}[{\rm (a)}]
\item A \emph{fan} in the sense of toric geometry (\cite[Ch.10]{Kripo}, \cite[Ch.3]{Toric-varieties}) is a conical complex. However, not all conical complexes are affine isomorphic to fans. Such an example is given, for instance, by $\C(\Pi_c)$, where $\Pi_c$ is as in Example \ref{complexes-examples}.
\item There are conical complexes not affine-isomorphic to $\C(\Pi)$ for $\Pi$ polytopal. One can even find such examples among fans; see \cite[Exercise 1.24]{Kripo}.
\item \emph{Projective fans}, i.e., those defining projective toric varieties, are affine-isomorphic to $\C(\Pi)$ for $\Pi$ polytopal. However, there are many non-projective fans which are affine-isomorphic to $\C(\Pi)$, $\Pi$ polytopal. For instance, one easily shows that every simplicial fan, whether or not projective, is affine-isomorphic to a projective simplicial fan, whereas in high dimensions the projective simplicial fans constitute a tiny fraction of all complete simplicial fans \cite{Pfeifle}.
\end{enumerate}
\end{example}

We want to introduce a natural tensor product of polytopal complexes, extending the notion for single polytopes. We do this by first introducing the tensor product for conical complexes and then descending to the polytopal case through the coning functor.

Let $\Pi_1$ and $\Pi_2$ be conical complexes and let $p_1,p_2\in\Pi_1$ and $q_1,q_2\in\Pi_2$ be faces. For the cones
\begin{align*}
C_i=P_{p_i},\qquad D_i=P_{q_i},\qquad&C'_i=\pi^{-1}_{p_i}(p_1\cap p_2)\subset C_i,\\
&D_i'=\pi^{-1}_{q_i}(q_1\cap q_2)\subset D_i,\\
&\quad\qquad\qquad\qquad\qquad\qquad i=1,2,
\end{align*}
Theorem \ref{folklore}(c) yields the face subcones
$$
C_i'\otimes D_i'\subset C_i\otimes D_i,\qquad i=1,2,
$$
both isomorphic to the cone
$$
\pi^{-1}_{p_1\cap p_2}(p_1\cap p_2)\otimes \pi^{-1}_{q_1\cap q_2}(q_1\cap q_2).
$$
Moreover, $C_1'\otimes D_1'$ is the largest of the faces of $C_1\otimes D_1$ of type $C'\otimes D'$, where
$$
C'=\pi_{p_1}^{-1}(p'),\quad D'=\pi_{q_1}^{-1}(q'),\quad p'\subset p_2,\quad q'\subset q_2,\quad p'\in\Pi_1,\quad q'\in\Pi;
$$
similarly for $C_2'\otimes D'_2$. It follows that the affine maps
\begin{align*}
{\big\{}
\xymatrix{
P_{p'}\otimes P_{q'}\ar[rrrr]^{(\pi^{-1}_p\circ\pi_{p'})\otimes(\pi^{-1}_q\circ\pi_{q'})}&&&&P_p\otimes P_q
}
\quad :\quad &p'\subset p,\qquad q'\subset q,\\
&p',p\in\Pi_1,\ q,q'\in\Pi_2{\big\}}
\end{align*}
glue together the tensor product cones in the way described in Definition \ref{polyhedral}. Consequently, the mentioned system of cones and affine maps can be augmented to a conical complex by adding the missing face cones of the tensor products, together with their face embeddings. (We do not know much of the missing face cones; see Theorem \ref{Basic}(c).) In order to have a full blown conical complex, we also need an abstract support space, built out of bijective images of the cones $P_p\otimes P_q$, and a compatible system of gluing bijections. We call the resulting conical complex the \emph{tensor product} of $\Pi_1$ and $\Pi_2$ and denote it by $\Pi_1\otimes\Pi_2$.

The facets and extremal rays of the complex $\Pi_1\otimes\Pi_1$ are naturally labeled by the symbols $p\otimes q$ where $p\in\Pi_1$ and $q\in\Pi_2$ are facets and rays, respectively. For the extremal rays here one uses Theorem \ref{folklore}(b). Consequently, if $\Pi_1'$ and $\Pi'_2$ are polytopal complexes then the facets (extremal rays) of the conical complex $\C(\Pi'_1)\otimes\C(\Pi'_2)$ can be naturally labeled by the symbols $p\otimes q$, where $p\in\Pi'_1$ and $q\in\Pi'_2$ are facets (respectively, vertices).

\medskip For two polytopes $P$ and $Q$ we have the \emph{degree map} $\deg:C(P)\otimes C(Q)\to\RR_+$, which is the linear extension of the assignment
$$
(a x,a)\otimes(b y,b)\mapsto ab,\quad x\in P,\quad y\in Q,\quad a,b\ge0.
$$
We have $P\otimes Q=\deg^{-1}(1)$ and $\C(P\otimes Q)=\C(P)\otimes\C(Q)$.

Now assume $\Pi_1$ and $\Pi_2$ are polytopal complexes. The degree map can be extended to their tensor product:
\begin{align*}
\deg_{\C}:|\C(\Pi_1)\otimes\C(\Pi_2)|\to\RR_+,\quad x\mapsto&\deg(\pi^{-1}_{p\otimes q}(x)),\\
&x\ \text{in the facet labelled by}\ p\otimes q,\\
&p\in\Pi_1\ \text{and}\ q\in\Pi_2\ \text{facets}.\\
\end{align*}

Finally, the \emph{tensor product of $\Pi_1$ and $\Pi_2$} is defined by the formula
$$
\Pi_1\otimes\Pi_2=\deg^{-1}_{\C}(1)\subset\C(\Pi_1)\otimes\C(\Pi_2).
$$
The following is immediate from the definition
\begin{equation}\label{formula-coning}
\C(\Pi_1)\otimes\C(\Pi_2)=\C(\Pi_1\otimes\Pi_2).
\end{equation}

\begin{remark}\label{Hom+}
\begin{enumerate}[{\rm (a)}]
\item For two polytopal complexes $\Pi_1$ and $\Pi_2$ we have
$$
\dim(\Pi_1\otimes\Pi_2)=\dim\Pi_1\dim\Pi_2+\dim\Pi_1+\dim\Pi_2.
$$
\item As a consequence of Theorem \ref{folklore}(j), if $\Delta_1$ and $\Delta_2$ are simplicial complexes then $\Delta_1\otimes\Delta_2$ is also a simplicial complex. If $\Gamma_1$ and $\Gamma_2$ are finite simple graphs, viewed as one-dimensional simplicial complexes, then $\Gamma_1\otimes\Gamma_2$ is a subcomplex of Babson-Kozlov's simplicial complex $\Hom_+(\Gamma_1,\Gamma_2)$, introduced in \cite{Lovasz-conjecture}.
\end{enumerate}
\end{remark}

\medskip Another and more straightforward construction for polytopal and conical complexes is their direct \emph{product} $\Pi_1\times\Pi_2$. It consists of the sets
$p\times q$ and maps $\pi_p\times\pi_q:P_p\times P_q\to p\times q$. That $\Pi_1\times\Pi_2$ is a genuine complex follows from the fact that one has total control over the faces of the direct product of polytopes.

For three polytopal or conical complexes $\Pi_1,\Pi_2,\Pi_3$, one has the following natural bijections of sets:
\begin{equation}\label{direct-product}
\begin{aligned}
&\Hom(\Pi_1\times\Pi_2,\Pi_3)\cong\Hom(\Pi_1,\Pi_3)\times\Hom(\Pi_2,\Pi_3),\\
&\Hom(\Pi_1,\Pi_2\times\Pi_3)\cong\Hom(\Pi_1,\Pi_2)\times\Hom(\Pi_1,\Pi_3).\\
\end{aligned}
\end{equation}

\medskip Let $\Pi_1$ and $\Pi_2$ be conical complexes and consider the map
\begin{align*}
&\Pi_1\times\Pi_2\Longrightarrow\Pi_1\otimes\Pi_2,\\
&\qquad\qquad\qquad(x,y)\mapsto\pi_{p\otimes q}(\pi^{-1}_p(x)\otimes\pi^{-1}_q(y)),\quad x\in p,\quad y\in q,\\
&\qquad\qquad\qquad\quad\qquad\qquad\qquad\qquad\qquad p\in\Pi_1\ \text{and}\ q\in\Pi_2\ \text{facets}.\\
\end{align*}

When $\Pi_1$ and $\Pi_2$ are polytopal complexes, the image of $|\Pi_1\times\Pi_2|$ under this map is in  $|\Pi_1\otimes\Pi_2|$. So we get a map $|\Pi_1\times\Pi_2|\to|\Pi_1\otimes\Pi_2|$, which will be denoted by the same $\Pi_1\times\Pi_2\Longrightarrow\Pi_1\otimes\Pi_2$.

In either case, conical or polytopal, the maps $\Pi_1\times\Pi_2\Longrightarrow\Pi_1\otimes\Pi_2$ are \emph{bi-affine} in the following sense: for any $x\in|\Pi_1|$ and $y\in|\Pi_2|$ the restrictions $\{x\}\times|\Pi_2|\to|\Pi_1\otimes\Pi_2|$ and $|\Pi_1|\times\{y\}\to|\Pi_1\otimes\Pi_2|$ are affine on the faces of $\Pi_1$ and $\Pi_2$, respectively.

\begin{lemma}\label{universal-biaffine}
Let $\Pi_1,\Pi_2,\Pi_3$ be either polytopal or conical complexes. Then the biaffine map $\Pi_1\times\Pi_2\Longrightarrow\Pi_1\otimes\Pi_2$ solves the following universal problem: any biaffine map $\Pi_1\times\Pi_2\Longrightarrow\Pi_3$ passes through a unique affine map $\phi$, making the following diagram commute
$$
\xymatrix{
\Pi_1\times\Pi_2\ar@{=>}[r]\ar@{=>}[rd]_f&\Pi_1\otimes\Pi_2\ar@{.>}[d]_{\circlearrowright\ \ }^{\exists!\phi}\\
&\Pi_3
}\ .
$$
Equivalently, we have a natural bijection of sets
\begin{equation}\label{tensor-hom}
\Hom(\Pi_1\otimes\Pi_2,\Pi_3)\cong\Hom(\Pi_1,\Hom(\Pi_2,\Pi_3)).
\end{equation}
In particular, the pairs of functors
\begin{align*}
&\otimes,\Hom:\Cones\times\Cones\to\Sets,\\
&\otimes,\Hom:\Poly\times\Poly\to\Sets
\end{align*}
form pairs of left and right adjoint functors.
\end{lemma}

\begin{proof}
When the complexes are conical, the corresponding linear algebra fact, applied to the facets of $\Pi_1$ and $\Pi_2$, yield affine maps from the facets of $\Pi_1\otimes\Pi_2$ to $\Pi_3$ and these maps patch together, yielding the desired map $\phi$. The uniqueness part also descends to the corresponding property for facets.  In view of the formula (\ref{formula-coning}), the polytopal case is a specialization of the conical one.
\end{proof}

\section{The hom-complex}\label{Hom-complex}

The main result of this section is

\begin{theorem}\label{Existence-theorem}
Let $\Pi_1$ and $\Pi_2$ be polytopal (conical) complexes. Then the space $\Hom(\Pi_1,\Pi_2)$ carries a polytopal (respectively, conical) complex structure, which can be defined algorithmically.
\end{theorem}

\begin{proof}
We consider only the polytopal case as the argument for conical complexes is verbatim the same.

For a polytopal complex $\Pi$ we let $\Lambda_\Pi$ denote the poset of symbols $\lambda_p$, $p\in\Pi$, ordered by $\lambda_p\le\lambda_q$ if and only if $p\subset q$. We view the set of monotone maps $\Lambda_{\Pi_1}\to\Lambda_{\Pi_2}$ as a poset with respect to the point-wise comparison.

Next we introduce the following correspondences between the monotone maps $\Psi:\Lambda_{\Pi_1}\to\Lambda_{\Pi_2}$ and elements of $\Hom(\Pi_1,\Pi_2)$.

\medskip\noindent(i) To $f\in\Hom(\Pi_1,\Pi_2)$ we associate the following monotone map $\Psi_f$. For $p_1\in\Pi_1$ let $\beta(P_{p_1})$ be the barycenter of $P_{p_1}=\pi_p^{-1}(p_1)$. There is a unique $p_2\in\Pi_2$ such that $(\pi_{p_2}^{-1}\circ f\circ\pi_{p_1})(\beta(P_{p_1}))\in\int(P_{p_2})$. We set $\Psi_f(\lambda_{p_1})=\lambda_{p_2}$. That $\Psi_f$ is a monotone map is straightforward.

\medskip\noindent(ii) Fix a monotone map $\Psi:\Lambda_{\Pi_1}\to\Lambda_{\Pi_2}$. We want to define a subset $[\Psi]\subset\Hom(\Pi_1,\Pi_2)$. This will be done in several steps.

For any pair $p\subset q$ in $\Pi_1$ we have the embedding
$$
\iota_{pq}:\Aff(P_p)\to\Aff(P_q),
$$
which is the affine extension of the map
$$
P_p\to P_q,\quad x\mapsto(\pi_q^{-1}\circ\pi_p)(x),\quad x\in P_p.
$$
For any $p\in\Pi_1$ we have the affine space
\begin{align*}
\AA_p(\Psi)=\bigcap_{
\tiny{
\begin{matrix}
r\in\Pi_1\\
r\subset p\\
\end{matrix}
}}
\big\{\phi\in\Aff\big(P_p,&\ P_{\Psi(p)}\big)\ :\\
&(\phi\circ\iota_{rp})\big(\Aff(P_r)\big)\subset\iota_{\Psi(r)\Psi(p)}\big(\Aff(P_{\Psi(r)})\big)\big\},
\end{align*}

\medskip\noindent(The set $\AA_p(\Psi)$ may well be empty.)

Then we form the diagram of affine maps and affine spaces
$$
\cD_{\Aff}(\Psi)=\{\rest.:\AA_q(\Psi)\to\AA_p(\Psi)\ :\ p\subset q,\quad p,q\subset\Pi_1\},
$$
where the maps `$\rest.$' are defined by the commutativity re\-qu\-i\-re\-ment for the squa\-res
\begin{equation}\label{numbered-square}
\xymatrix{
\Aff(P_q)\ar[rr]^\phi&&\Aff(P_{\Psi(q)})\\
&&\\
\Aff(P_p)\ar[uu]^{\iota_{pq}}\ar[rr]_{\rest.(\phi)}&&\Aff(P_{\Psi(p)})\ar[uu]_{\iota_{\Psi(p)\Psi(q)}}\\
}.
\end{equation}

\medskip Consider the diagram of polytopes and affine maps
$$
\cD_{\pol}(\Psi)=\{\rest.:R_q\to R_p\ :\ p\subset q,\quad p,q\subset\Pi_1\},
$$
where:
\begin{enumerate}[{\rm$\centerdot$}]
\item
$R_r=\AA_r(\Psi)\cap\Hom(P_r,P_{\Psi(r)})$ for $r\in\Pi_1$, the intersection being considered in $\Aff\big(P_p,P_{\Psi(p)}\big)$,
\item the maps are determined by the commutativity condition for the squares
$$
\xymatrix{
P_q\ar[rr]^\phi&&P_{\Psi(q)}\\
&&\\
P_p\ar[uu]^{\iota_{pq}}\ar[rr]_{\rest.(\phi)}&&P_{\Psi(p)}\ar[uu]_{\iota_{\Psi(p)\Psi(q)}}\\
}.
$$
\end{enumerate}
The existence of the squares above follows from (\ref{numbered-square}) and the equalities $P_{\Psi(p)}=P_{\Psi(q)}\cap\Aff(P_{\Psi(p)})$.

Finally, the polytope $[\Psi]$ is defined by
$$
[\Psi]:=\lim_\leftarrow\mathcal D_{\pol}(\Psi)\subset\lim_\leftarrow\mathcal D_{\Aff}(\Psi).
$$
The set $[\Psi]$ is naturally thought of as a subset of $\Hom(\Pi_1,\Pi_2)$: an element $f\in[\Psi]$ means a collection of affine maps $P_p\to P_{\Psi(p)}$, compatible with the structural maps $\pi$-s in $\Pi_1$ and $\Pi_2$. (It is possible that $[\Psi]=\emptyset$ for some monotone maps $\Psi$, even if $\lim_\leftarrow\mathcal D_{\Aff}(\Psi)\not=\emptyset$.)

\medskip\noindent\emph{Notice.} Formally speaking, the set $[\Psi]$ embeds into $\Hom(\Pi_1,\Pi_2)$ via the structural maps $\pi$. But  in order not to overload notation, we think of $[\Psi]$ as its bijective image under this embedding. A more explicit polytopal description of this limit is given in Section \ref{Inverse}.

\medskip Next we observe that the polytopes $[\Psi]$ cover the whole set $\Hom(\Pi_1,\Pi_2)$ and they patch together, forming a polytopal complex. The first claim follows from the equality $f\in[\Psi_f]$ for every $f\in\Hom(\Pi_1,\Pi_2)$. The second claim follows from the equality $[\Psi_1]\cap[\Psi_2]=[\Psi_1\wedge\Psi_2]$ for any two monotone maps $\Psi_1,\Psi_2:\Lambda_{\Pi_1}\to\Lambda_{\Pi_2}$ and the fact that $[\Psi_1]$ is a face of $[\Psi_2]$ when $\Psi_1\le\Psi_2$.

In order describe the face poset structure of $\Hom(\Pi_1,\Pi_2)$, we observe that for every $f\in\Hom(\Pi_1,\Pi_2)$ the map $\Psi_f$ is the smallest among the monotone maps $\Psi:\Lambda_{\Pi_1}\to\Lambda_{\Pi_2}$ for which $[\Psi]=[\Psi_f]$. So the poset in question is the poset of monotone maps
$$
\Lambda_{\Pi_1,\Pi_2}:=\{\Psi_f\ :\ f\in\Hom(\Pi_1,\Pi_2)\}.
$$

Detailed analysis of the algorithmic aspects of the constructions above is deferred to Section \ref{Implementing}.
\end{proof}

As a combined effect of Lemma \ref{universal-biaffine} and Theorem \ref{Existence-theorem}, we have

\begin{corollary}\label{conjunction}
The bijections of sets (\ref{direct-product}) and (\ref{tensor-hom}) are affine isomorphisms of complexes. Both categories $\Poly$ and $\Cones$ are symmetric monoidal closed categories, enriched over themselves.
\end{corollary}

For the categorial terminology used above, see \cite{Enriched}. One needs the \emph{symmetry}, \emph{pentagon coherence}, and \emph{hexagon coherence} properties of the bifunctor $\otimes$. These are inherited from the similar properties of the tensor product of vector spaces. We do not delve into the definitions because the polytopal contents is contained in the first part of the corollary.

\section{The algorithm for $\Hom(\Pi_1,\Pi_2)$}\label{Implementing}

Here we discuss an algorithm for computing $\Hom(\Pi_1,\Pi_2)$, resulting from the proof of Theorem \ref{Existence-theorem}, making shortcut whenever possible. The notation throughout this section is the same as in that theorem and its proof.

We also remark that, since the order complex of the poset $\Lambda_{\Pi_1,\Pi_2}$ introduced at the end of the proof of Theorem \ref{Existence-theorem} is the barycentric subdivision of $\Hom(\Pi_1,\Pi_2)$, one can compute the integer homology $H_*(|\Hom(\Pi_1,\Pi_2)|,\ZZ)$ by computing $H_*(\Lambda_{\Pi_1,\Pi_2},\ZZ)$, for which an existing platform is \cite{homology}.

\subsection{Complexes succinctly}\label{Optimal}

In the algorithmic version of Theorem \ref{Existence-theorem} one eliminates the reference to abstract sets $p$ and works directly with polytopes. Definition \ref{polyhedral} is modeled after the definition of $CW$-complexes -- it puts the main emphasis on the support spaces and makes easier to work with in the proof of Theorem \ref{Existence-theorem}. The equivalent definition, where all polytopes and affine maps are storable as matrices and vectors, is as follows. A lattice polyhedral complex $\Pi$ consists of:
\begin{enumerate}[{\rm (a)}]
\item a finite poset $\Lambda$,
\item
a collection of nonempty polytopes $\{P_\lambda\subset\RR^{d_\lambda}\}_\Lambda$,
\item affine maps $\iota_{\lambda\mu}:P_\lambda\to Q_\mu$ whenever $\lambda\le\mu$, mapping $P_\lambda$ isomorphically onto a face of $Q_\mu$.
\end{enumerate}
Furthermore we require the following compatibility conditions:
\begin{enumerate}[{\rm (i)}]
\item  $\iota_{\lambda\lambda}={\bf1}_{P_\lambda}$ and $\iota_{\mu\nu}\circ\iota_{\lambda\mu}=\iota_{\lambda\nu}$ for $\lambda,\mu\in\Lambda$,
\item
for every element $\lambda\in\Lambda$ and each face $F$ of the polytope $P_\lambda$ there
is a unique element $\mu\in\Lambda$ such that $\mu\le\lambda$ and $\iota_{\mu\lambda}(P_\mu)=F$.
\end{enumerate}

Explicating the diagrams $\cD_{\Aff}(\Psi)$ and $\cD_{\pol}(\Psi)$ involves such procedures as forming the affine hulls of polytopes and solving systems of linear equalities. The polytopes $\Hom(P_r,P_{\Psi(r)})$ can be computed using \textsf{Polymake}.

\medskip The first essential speedup of the algorithm can be achieved by restricting to the \emph{continuous monotone maps} $\Psi:\Lambda_{\Pi_1}\to\Lambda_{\Pi_2}$, i.e., the monotone maps, satisfying the condition
\begin{align*}
\Psi(\lambda_p)=\Psi(\lambda_{q_1})\vee\cdots\vee\Psi(\lambda_{q_k}),\quad\text{where}\ \ \lambda_p&=\lambda_{q_1}\vee\cdots\vee\lambda_{q_k},\\
&q_1,\ldots,q_k\in\Pi_1,\ \ k\in\NN,\ \ p\in\Pi_1,
\end{align*}
or, equivalently,
\begin{align*}
\Psi(\lambda_p)=\bigvee_{v\in\vertex(p)}\Psi(\lambda_v),\qquad p\in\Pi_1.
\end{align*}

The second substantial speedup is based on the following observation. Call a face of a polytopal complex \emph{essential} if it is the intersection of a family of facets; e.g., the facets are essential faces. For most polytopal complexes the essential faces constitute only a small part of all faces; admittedly, this is not true when the support space is a topological manifold. Yet the essential faces often suffice for computational purposes. This is the case when one describes $\Hom(\Pi_1,\Pi_2)$ in terms of $\Pi_1$ and $\Pi_2$. More precisely, the proof of Theorem \ref{Existence-theorem} goes through if the posets $\Lambda_{\Pi_1}$ and $\Lambda_{\Pi_2}$ are changed to their \emph{essential} sub-posets, i.e., the ones which correspond to the essential faces of $\Pi_1$ and $\Pi_2$. In fact, if one requires the condition (a) in Definition \ref{polyhedral} only for the essential faces of $P_p$ and adjusts Definition \ref{complex-maps} accordingly, then one obtains a category isomorphic to $\Poly$. In particular, in the algorithm for $\Hom(\Pi_1,\Pi_2)$, one can restrict to the continuous monotone maps between the essential sub-posets of $\Lambda_{\Pi_1}$ and $\Lambda_{\Pi_2}$. In the extremal case when $\Pi_1$ and $\Pi_2$ are the faces of single polytopes, these are one-point posets and the algorithm becomes the computation of the hom-polytope.

\subsection{Limits succinctly}\label{Inverse}
Denote the essential sub-posets of $\Lambda_{\Pi_1}$ and $\Lambda_{\Pi_2}$ by $\ess(\Lambda_{\Pi_1})$ and $\ess(\Lambda_{\Pi_2})$, respectively.

Effective computation of the limit of a diagram in $\Pol$ is a challenge of independent interest -- and so is the computation of colimits! Our diagram $\cD_{\pol}(\Psi)$ is special though: it is a (covariant) functor from the opposite poset $\ess(\Lambda_{\Pi_1})^{\op}$ to $\Pol$. So the limit allows the following succinct description.

Let $\lambda_1,\ldots,\lambda_l$ be the maximal elements of $\Lambda_{\Pi_1}$, i.e., they correspond to the facets of $\Pi_1$. Assume
\begin{align*}
\{\mu_1,\ldots,\mu_m\}=\{\lambda_i\wedge\lambda_j\ :\ i\not=j\ \text{and the meet exists}\}_{i,j=1}^l.
\end{align*}
(By definition the faces of polytopal complexes are non-empty, making possible the non-existence of some of infima.)

Let $\Lambda^{\o}$ denote the sub-poset of $\ess(\Lambda_{\Pi_1})^{\op}$, consisting of the elements $\lambda_1,\ldots,\lambda_l$ and $\mu_1,\ldots,\mu_m$. Denote by $\cD_{\Aff}(\Psi)^{\o}$ and $\cD_{\pol}(\Psi)^{\o}$ the corresponding restrictions to $\Lambda^{\o}$. Then we have the equalities
$$
\lim_{\leftarrow}\cD_{\Aff}(\Psi)^{\o}=\lim_{\leftarrow}\cD_{\Aff}(\Psi)\qquad\text{and}\qquad[\Psi]=\lim_{\leftarrow}\cD_{\pol}(\Psi)^{\o}.
$$

In explicit terms, if $p_1,\ldots,p_l$ are the facets of $\Pi_1$ and $q_1,\ldots,q_m$ are the non-empty pairwise intersections of the $p_i$, then
\begin{align*}
[\Psi]=\big\{(x_1,\ldots,x_l)\ :\ \rest.(x_i)=\rest.(x_j),\quad&\quad
\xymatrix{
R_{p_i}\ar[dr]_{\rest.}&& R_{p_j}\ar[dl]^{\rest.}\\
&R_{q_k}&
},\\
&q_k=p_i\cap p_j,\quad i,j=1,\ldots,l,\ i\not=j\big\}\\
\\
&\subset R_{p_1}\times\cdots\times R_{p_l}.\\
\end{align*}
So $[\Psi]$ is the solution set to a relatively small system of linear equations.

\subsection{Euclidean complexes}\label{Euclidean-complexes} Here we explain how the determination of the hom-complex between Euclidean complexes admits a substantial simplification. This is based on the following

\begin{proposition}\label{Free-embedding}
For any Euclidean complex $\Pi$ there exists $m\in\NN$ and an embedding $\iota:|\Pi|\to\RR^m$, satisfying the conditions:
\begin{enumerate}[{\rm (a)}]
\item
$\iota\circ\pi_p:P_p\to\RR^m$ is affine for every face $p\in\Pi$,
\item for every Euclidean complex $R$ in $\RR^n$ and an affine map $f:\Pi\to R$ there is a unique affine map $\phi:\RR^m\to\RR^n$ with $f=\phi\circ\iota$.
\end{enumerate}
\end{proposition}

\begin{proof}
Let $\RR^d$ be the ambient vector space for $\Pi$. We will identify the faces $p\in\Pi$ with the polytopes $P_p\subset\RR^d$ along the bijections $\pi_p$.

Denote
$$
\AA_{\Pi}:=\lim_{\to}\big(\Aff(p)\hookrightarrow\Aff(q)\ :\ p\subset q,\ p,q\in\Pi\big),
$$
where the colimit is taken in the category of affine spaces and affine maps, and consider the maps
$$
\iota_p:\Aff(p)\to\AA_{\Pi},\qquad p\in\Pi.
$$

Because of the commutative diagram of affine spaces and inclusion maps
$$
\xymatrix{
&\RR^d&&\\
\Aff(p)\ar[ur]\ar[rr]&&\Aff(q),\ar[ul]&p\subset q,\quad p,q\in\Pi,\\
}
$$
there is a unique affine map $\psi:\AA_{\Pi}\to\RR^d$, such that $\psi\circ\iota_p$ is the inclusion map $\Aff(p)\hookrightarrow\RR^d$ for every face $p$. This implies that the maps $\iota_p$ are all injective and, also, $\iota_p(p)\cap\iota_q(q)$ is a face both of $p$ and $q$ for $p,q\in\Pi$.

We can think of $\AA_{\Pi}$ as $\RR^m$ for $m=\dim\AA_{\Pi}$. In particular, we have the injective map $\iota:|\Pi|\to\RR^m$, defined by $\iota(x)=\iota_p(x)$ for $x\in p$, $p\in\Pi$. This map is affine on the faces of $\Pi$.

Let $R$ be a Euclidean complex in $\RR^n$ and $f:\Pi\to R$ be an affine map. As the case of $\Pi$, we identify the faces of $R$ with their polytopal preimages in $\RR^n$ along the structural bijections. For an affine map $f:\Pi\to R$ we have the diagram of affine spaces and affine maps
$$
\xymatrix{
&\RR^n&&\\
\Aff(p)\ar[ur]^{\Aff(f|_p)}\ar[rr]&&\Aff(q),\ar[ul]_{\Aff(f|_q)}&p\subset q,\quad p,q\in\Pi,\\
}
$$
where: $\Aff(f|_p)$ is the affine extension to $\Aff(p)$ of the restriction $f|_p$, composed with the inclusion into $\RR^n$, and similarly for $\Aff(f|_q)$, whereas the horizontal arrows represent the inclusion maps. Since $\RR^m$ is colimit, there is an affine map $\phi:\RR^m\to\RR^n$ for which $\phi\circ\iota=f$. If there was another affine map $\phi':\RR^m\to\RR^n$ with $\phi'\circ\iota=f$, then $\phi\circ\iota$ and $\phi'\circ\iota$ would coincide on the affine hulls $\Aff(p)$ for all $p\in\Pi$, contradicting the universality of $\RR^m$.
\end{proof}

\noindent\emph{Back to the algorithm:} for a Euclidean complex $\Pi_1$, the proof of Proposition \ref{Free-embedding} suggests an algorithm for constructing an embedding $\iota:|\Pi_1|\to\RR^m$ as in the proposition. (This step requires effective colimit computations.) Assume $\Pi_2$ is a Euclidean complex in $\RR^n$. Then $\Hom(\Pi_1,\Pi_2)$ can be thought of as a subset of $\Aff(\RR^m,\RR^n)$. Further, we can think of the affine maps $\RR^m\to\RR^n$ as pairs $(M,v)$, where $M$ is a $n\times m$ matrix and $v\in\RR^n$.

For every continuous monotone map $\Psi:\ess(\Lambda_{\Pi_1})\to\ess(\Lambda_{\Pi_2})$ the condition $(M,v)\in[\Psi]$ directly translates into linear constraints on the entries of $M$ and $v$. This is a deep shortcut in the determination process of the polytope $[\Psi]$, granted the embedding $\iota:|\Pi_1|\to\RR^m$ is already computed. (One still needs the fact that the polytopes $[\Psi]$ define a polytopal complex structure on $\Hom(\Pi_1,\Pi_2)$, which is the contents of Theorem \ref{Existence-theorem}.)

\subsection{Simplicial complexes}\label{Simplicial-complexes} Let $\Delta_1$ and $\Delta_2$ be simplicial complexes and consider the set of maps
\begin{align*}
\Lambda_{\Delta_1,\Delta_2}=\{\alpha:\vertex(\Delta_1)\to\Delta_2\ :\ &\alpha(v)\not=\emptyset\ \ \text{for all}\ \ v\in\vertex(\Delta_1)\ \text{and}\\
&\qquad\bigcup_{v\in\tau}\alpha(v)\in\Delta_2\ \text{for every}\ \tau\in\Delta_1\},
\end{align*}
made into a poset by letting $\alpha\le\beta$ if and only if $\alpha(v)\subset\beta(v)$ for all $v\in\vertex(\Delta_1)$. 

The poset $\Lambda_{\Delta_1,\Delta_2}$ is formally different from $\Lambda_{\Pi_1,\Pi_2}$, introduced at the end of the proof Theorem \ref{Existence-theorem}, when $\Pi_1=\Delta_1$ and $\Pi_2=\Delta_2$. But, as it follows from the discussion below, the two are naturally isomorphic.

For every element $\alpha\in\Lambda_{\Delta_1,\Delta_2}$ we have the product polytope
$$
\Box_\alpha=\prod_{v\in\vertex(\Delta_1)}\conv(\alpha(v))\subset|\Delta_2|^{\#\vertex(\Delta_1)}.
$$
One has $\alpha\le\beta$ if and only if $\Box_\alpha$ is a face of $\Box_\beta$ and all faces of $\Box_\beta$ arise this way.

Next we observe that each of the polytopes $\Box_\alpha$ is a product of simplices. The easiest way to see this is via observing the equality for the corresponding relative interiors:
\begin{align*}
\int(\Box_\alpha)=\prod_{v\in\vertex(\Delta_1)}\int(\alpha(v)).
\end{align*}
Moreover, every element of $\int(\Box_\alpha)$ is naturally interpreted as an affine map $\Delta_1\to\Delta_2$. All these observations are based by the defining property of a simplex that every map from its vertices to a polytope (in our situation, a simplex) extends uniquely to an affine map from the simplex.

As $\alpha$ varies, the polytopes $\Box_\alpha$ patch up to a polytopal complex whose face lattice is given by $\Lambda_{\Delta_1,\Delta_2}$. The proof of Theorem \ref{Existence-theorem} implies the part (a) of the following proposition, and the other parts are direct consequences:

\begin{proposition}\label{Existence-simplicial}
In the notation introduced above, we have:
\begin{enumerate}[{\rm (a)}]
\item $\Hom(\Delta_1,\Delta_2)$ is the polytopal complex, obtained by gluing the products of simplices $\Box_\alpha$ ($\alpha\in\Lambda_{\Delta_1,\Delta_2}$) along common faces as induced by the poset $\Lambda_{\Delta_1,\Delta_2}$.
\item An affine map $f:\Delta_1\to\Delta_2$ is a vertex of $\Hom(\Delta_1,\Delta_2)$ if and only if $f$ is a simplicial map.
\item If $\dim\Delta_2=1$ then $\Hom(\Delta_1,\Delta_2)$ is cubical complex.
\end{enumerate}
\end{proposition}
Above, the affine maps $\Delta_1\to\Delta_2$, mapping the vertices to vertices, are called \emph{simplicial maps}.

$\Hom(\Delta_1,\Delta_2)$ is a subcomplex of $\Hom_M(\vertex(\Delta_1),\vertex(\Delta_2))$, introduced by Kozlov \cite[p.143]{Kozlov}, where $M$ is the set of simplicial maps $\Delta_1\to\Delta_2$. The faces of $\Hom_M(\vertex(\Delta_1),\vertex(\Delta_2))$ are the products $\prod_{\vertex(\Delta_1)}\sigma_v$, where the factors $\sigma_v$ are the simplices satisfying the conditions: (i) $\vertex(\sigma_v)\subset\vertex(\Delta_2)$ for every $v$ and (ii) any map $\phi:\vertex(\Delta_1)\to\vertex(\Delta_2)$ with $\phi(v)\in\sigma_v$ for all $v$ defines a simplicial map $\Delta_1\to\Delta_2$. Such $\sigma_v$ may \emph{not} be a simplex in $\Delta_2$, not even when $\dim\Delta_1=1$. In fact, when $\dim\Delta_1=1$ the condition on the factors $\sigma_v$ just says that for any two distinct vertices $v,w\in\vertex(\Delta_1)$, connected by an edge in $\Delta_1$, the vertices of $\sigma_v$ are connected with those of $\sigma_w$ by edges in $\Delta_2$ -- a weaker condition than the requirement that $\sigma_v\cup\sigma_w\in\Delta_2$, used in the description of $\Hom(\Delta_1,\Delta_2)$

If $\Gamma_1$ and $\Gamma_2$ are simple graphs, viewed as one-dimensional simplicial complexes, then $\Hom_M(\vertex(\Delta_1),\vertex(\Delta_2))$ is known as \emph{Lovasz' complex} $\Hom(\Gamma_1,\Gamma_2)$ \cite[Definition 9.23]{Kozlov}. Notice that, if $\mathring\Gamma_1$ and $\mathring\Gamma_2$ are obtained from $\Gamma_1$ and $\Gamma_2$ by adding one loop per a vertex, then the affine hom-complex $\Hom(\Gamma_1,\Gamma_2)$ coincides with Lovasz' $\Hom(\mathring\Gamma_1,\mathring\Gamma_2)$. 

\bibliography{references}
\bibliographystyle{plain}

\end{document}